\documentclass[12pt,a4paper]{article}
\usepackage{amsmath}
\usepackage{amssymb}
\usepackage{amscd}        
\usepackage{graphics}
\usepackage{color}
\usepackage[all]{xy}
\usepackage{amsthm}   
\theoremstyle{plain}
\newtheorem{theorem}{Theorem}[section]
\newtheorem{lemma}[theorem]{Lemma}
\newtheorem{corollary}[theorem]{Corollary}
\newtheorem{proposition}[theorem]{Proposition}
\theoremstyle{definition}

\newtheorem{definition}[theorem]{Definition}
\newtheorem{example}[theorem]{Example}
\theoremstyle{remark}

\newcommand{\coclosed}[2]{\ensuremath{\xymatrix@1{ #1 \, \ar@{^{(}->}[r]^-{cc} &
#2}}}
\newcommand{\cosmall}[3]{\ensuremath{\xymatrix@1{ #1 \,
\ar@{^{(}->}[r]^-{cs}_-{#2} & #3}}}

\usepackage[none]{hyphenat}
\author{Ay\c{s}e Tu\u{g}ba G\"{u}ro\u{g}lu \footnote{Celal Bayar University, Faculty of Arts and Sciences, Department of Mathematics, Muradiye, Manisa, Turkey.\newline e-mail : tugba.guroglu@cbu.edu.tr} \quad and \quad Elif Tu\u{g}\c{c}e Meri\c{c} \footnote{Celal Bayar University, Faculty of Arts and Sciences, Department of Mathematics, Muradiye, Manisa, Turkey.\newline e-mail : tugce.meric@cbu.edu.tr}}

\title{Principally Goldie*-Lifting Modules}
\date{}
\begin{document}
\maketitle
\renewcommand{\theenumi}{\arabic{enumi}}
\renewcommand{\labelenumi}{\emph{(\theenumi)}}

\textbf{ABSTRACT.} A module $M$ is called principally Goldie*-lifting if for every proper cyclic submodule $X$ of $M$, there is a direct summand $D$ of $M$ such that $X\beta^* D$. In this paper, we focus on principally Goldie*-lifting modules as generalizations of lifting modules. Various properties of these modules are given.

\vspace{1cm}
\begin{flushleft}
\end{flushleft}

\section{Introduction}
\quad Throughout this paper $R$ denotes an associative ring with identity and all modules are unital right $R$-modules. $Rad(M)$ will denote the Jacobson radical of $M$. Let $M$ be an $R$-module and $N,K$ be submodules of $M$. A submodule $K$ of a module $M$ is called \emph{small} (or superfluous) in $M$, denoted by $K\ll M$, if for every submodule $N$ of $M$ the equality $K+N=M$ implies $N=M$. $K$ is called a \emph{supplement} of $N$ in $M$ if $K$ is a minimal with respect to $N+K=M$, equivalently $K$ is a supplement (weak supplement) of $N$ in $M$ if and only if $K+N=M$ and $K\cap N\ll K$ ($K\cap N\ll M$). A module $M$ is called \emph{supplemented module} (\emph{weakly supplemented module}) if every submodule of $M$ has a supplement (weak supplement) in $M$. A module $M$ is $\oplus$-\emph{supplemented module} if every submodule of $M$ has a supplement which is a direct summand in $M$. \cite{Harmancý} defines principally supplemented modules and investigates their properties. A module $M$ is said to be \emph{principally supplemented} if for all cyclic submodule $X$ of $M$ there exists a submodule $N$ of $M$ such that $M=N+X$ with $N\cap X$ is small in $N$.  A module $M$ is said to be \emph{$\oplus$-principally supplemented} if for each cyclic submodule $X$ of $M$ there exists a direct summand $D$ of $M$ such that $M=D+X$ and $D\cap X\ll D$. A nonzero module $M$ is said to be \emph{hollow} if every proper submodule of $M$ is small in $M$. A nonzero module $M$ is said to be \emph{principally hollow} which means every proper cyclic submodule of $M$ is small in $M$. Clearly, hollow modules are principally hollow. Given submodules $K\subseteq N \subseteq M$ the inclusion $K\overset{cs}\hookrightarrow N$ is called \emph{cosmall }in $M$, denoted by $K\hookrightarrow N$, if $N/K \ll M/K$.\\
\newline \quad Lifting modules play an important role in module theory. Also their various generalizations are studied by many authors in \cite{Harmancý}, \cite{Býrkenmeýer}, \cite{Kamal}, \cite{Keskin}, \cite{Muller}, \cite{Talebý}, \cite{Yongduo} etc. A module $M$ is called \emph{lifting module} if for every submodule $N$ of $M$ there is a decomposition $M=D\oplus D'$ such that $D\subseteq N$ and $D'\cap N \ll M$. A module $M$ is called \emph{principally lifting module} if for all cyclic submodule $X$ of $M$ there exists a decomposition $M=D\oplus D'$ such that $D\subseteq X$ and $D'\cap X \ll M$. G.F.Birkenmeier et.al. \cite{Býrkenmeýer} defines $\beta^*$ relation to study on the open problem `Is every $H$-supplemented module supplemented?' in \cite{Muller}. They say submodules $X$, $Y$ of $M$ are $\beta^*$ equivalent, $X\beta^* Y$, if and only if $\dfrac{X+Y}{X}$ is small in $\dfrac{M}{X}$ and $\dfrac{X+Y}{Y}$ is small in $\dfrac{M}{Y}$.  $M$ is called \emph{Goldie*-lifting }(or briefly, \emph{$\mathcal{G}$*-lifting}) if and only if for each $X\leq M$ there exists a direct summand $D$ of $M$ such that $X\beta^*D$. $M$ is called \emph{Goldie*-supplemented} (or briefly, \emph{$\mathcal{G}$*-supplemented}) if and only if for each $X\leq M$ there exists a supplement submodule $S$ of $M$ such that $X\beta^*S$ (see \cite{Býrkenmeýer}).

In Section $2$, we recall the equivalence relation $\beta^*$ which is defined in  \cite{Býrkenmeýer} and investigate some basic properties of it.

In section $3$ we define principally Goldie*-lifting modules as a generalization of lifting modules. We give some neccesary assumptions for a quotient module or a direct summand of a principally Goldie*-lifting module to be principally Goldie*-lifting. Principally lifting, principally Goldie*-lifting and principally supplemented modules are compared. It is also shown that  principally lifting, principally Goldie*-lifting and $\oplus$-principally supplemented coincide on $\pi$-projective modules.

\section{Properties of $\beta^*$ Relation }
\begin{definition}(See \cite{Býrkenmeýer}) \label{rel}
Any submodules $X,Y$ of $M$ are $\beta^*$ equivalent, $X\beta^*Y$, if and only if $\dfrac{X+Y}{X}$ is small in $\dfrac{M}{X}$ and $\dfrac{X+Y}{Y}$ is small in $\dfrac{M}{Y}$.
\end{definition}

\begin{lemma}\normalfont{(See \cite{Býrkenmeýer})}
$\beta^*$ is an equivalence relation.
\end{lemma}

By [\cite{Býrkenmeýer}, page $43$], the zero submodule is $\beta^*$ equivalent to any small submodule.

\begin{theorem}\normalfont{(See \cite{Býrkenmeýer})} \label{thm1}
Let $X,Y$ be submodules of $M$. The following are equivalent:
\begin{itemize}
  \item[(a)] $X\beta^*Y$.
  \item[(b)] $X\overset{cs}\hookrightarrow X+Y$  and $Y\overset{cs}\hookrightarrow X+Y$.
  \item[(c)] For each submodule $A$ of $M$ such that $X+Y+A=M$, then $X+A=M$ and $Y+A=M$.
  \item[(d)] If $K+X=M$ for any submodule $K$ of $M$, then $Y+K=M$ and if $Y+H=M$ for any submodule $H$ of $M$, then $X+H=M$.
\end{itemize}
\end{theorem}

\begin{lemma} \label{lemma1}
Let $M=D\oplus D'$ and $A,B\leq D$. Then $A\beta$*$B$ in $M$ if and only if  $A\beta$*$B$ in $D$.
\end{lemma}
\begin{proof}
$(\Rightarrow)$ Let $A\beta$*$B$ in $M$ and $A+B+N=D$ for some submodule $N$ of $D$. Let us show $A+N=D$ and $B+N=D$. Since $A\beta$*$B$ in $M$, $$M=D\oplus D'=A+B+N+D'$$ implies $A+N+D'=M$ and $B+N+D'=M$. By [\cite{Wisbauer}, $41$], $A+N=D$ and $B+N=D$. From Theorem \ref{thm1}, we get $A\beta$*$B$ in $D$.\\
$(\Leftarrow)$ Let $A\beta$*$B$ in $D$. Then $\dfrac{A+B}{A}\ll \dfrac{D}{A}$ implies $\dfrac{A+B}{A}\ll \dfrac{M}{A}$. Similarly, $\dfrac{A+B}{B}\ll \dfrac{D}{B}$ implies $\dfrac{A+B}{B}\ll \dfrac{M}{B}$. This means that $A\beta$*$B$ in $M$.
\end{proof}

\begin{lemma} \label{lemma2}
If a direct summand $D$ in $M$ is $\beta$* equivalent to a cyclic submodule $X$ of $M$, then $D$ is also cyclic.
\end{lemma}
\begin{proof}
Assume that $M=D\oplus D'$ for some submodules $D,D'$ of $M$ and  $X$ is a cyclic submodule of $M$ which is $\beta$* equivalent to $D$. By Theorem \ref{thm1},  $M=X+D'$. Since $\dfrac{X+D'}{D'}=\dfrac{M}{D'}\cong D$, $D$ is cyclic.
\end{proof}

\section{Principally Goldie* - Lifting Modules}
\indent \indent In \cite{Býrkenmeýer}, the authors defined $\beta^*$ relation and they introduced two notions called Goldie*-supplemented module and Goldie*-lifting module depend on $\beta^*$ relation. $M$ is called \emph{Goldie*-lifting} (or briefly, $\mathcal{G}$*-lifting) if and only if for each $N\leq M$ there exists a direct summand $D$ of $M$ such that $N\beta^*D$. $M$ is called \emph{Goldie*-supplemented} (or briefly, $\mathcal{G}$*-supplemented) if and only if for each $N\leq M$ there exists a supplement submodule $S$ of $M$ such that $N\beta^*S$. A module $M$ is said to be \emph{$H$-supplemented }if for every submodule $N$ there is a direct summand $D$ of $M$ such that $M=N+B$ holds if and only if $M=D+B$ for any submodule $B$ of $M$. They showed that Goldie*-lifting modules and $H$-supplemented modules are the same in [\cite{Býrkenmeýer}, Theorem $3.6$]. In this section, we define principally Goldie*-lifting module (briefly principally $\mathcal{G}$*-lifting module) as a generalization of $\mathcal{G}$*-lifting module and investigate some properties of this module.

\begin{definition}
A module $M$ is called \emph{principally Goldie*-lifting module} (briefly principally $\mathcal{G}$*-lifting) if for each cyclic submodule $X$ of $M$, there exists a direct summand $D$ of $M$ such that $X\beta^*D$.
\end{definition}

Clearly, every $\mathcal{G}$*-lifting module is principally $\mathcal{G}$*-lifting. However the converse does not hold.

\begin{example}
Consider the $\mathbb{Z}$-module $\mathbb{Q}$. Since $Rad (\mathbb{Q})= \mathbb{Q}$, every cyclic submodule of $\mathbb{Q}$ is small in $\mathbb{Q}$. By [\cite{Býrkenmeýer}, Example $2.15$], the $\mathbb{Z}$-module $\mathbb{Q}$ is principally $\mathcal{G}$*-lifting. But the $\mathbb{Z}$-module $\mathbb{Q}$ is not supplemented. So it is not $\mathcal{G}$*-lifting by [\cite{Býrkenmeýer}, Theorem $3.6$].
\end{example}

A module $M$ is said to be \emph{radical} if $Rad(M)=M$.

\begin{lemma}
Every radical module is principally $\mathcal{G}$*-lifting.
\end{lemma}
\begin{proof}
Let $m\in M$. If $M$ is radical, $mR \subseteq Rad(M)$. By [\cite{Wisbauer}, $21.5$], $mR\ll M$. So we get $mR \beta^*0$. Thus $M$ is principally $\mathcal{G}$*-lifting.
\end{proof}


\begin{theorem}\label{thm3}
Let $M$ be a module. Consider the following conditions:
\begin{itemize}
  \item[(a)] $M$ is principally lifting,
  \item[(b)] $M$ is principally $\mathcal{G}$*-lifting,
  \item[(c)] $M$ is principally supplemented.
\end{itemize}
Then $(a)\Rightarrow (b)\Rightarrow (c)$.
\end{theorem}
\begin{proof}
$(a)\Rightarrow (b)$ Let $m\in M$. Then $mR$ is cyclic submodule of $M$. From $(a)$, there is a decomposition $M=D\oplus D'$ with $D\leq mR$ and $mR \cap D' \ll M$. Since $D\leq mR$, $\dfrac{mR+D}{mR}\ll \dfrac{M}{mR}$. By modularity, $mR=D\oplus (mR\cap D')$. Then $\dfrac{mR}{D}\cong mR\cap D'$ and $\dfrac{M}{D}\cong D'$. If $mR \cap D' \ll M$, by [\cite{Wisbauer}, $19.3$], $mR \cap D' \ll D'$. It implies that $\dfrac{mR+D}{D}\ll \dfrac{M}{D}$. Therefore it is seen that $mR\beta$*$D$ from Definition \ref{rel}. Hence $M$ is principally $\mathcal{G}$*-lifting.\\
$(b)\Rightarrow (c)$ Let $m \in M$. By hypothesis, there exists a direct summand $D$ of $M$ such that $mR\beta$* $D$. Since $M=D\oplus D'$ for some submodule $D'$ of $M$ and $D'$ is a supplement of $D$, $D'$ is a supplement of $mR$ in $M$ by [\cite{Býrkenmeýer}, Theorem $2.6 (ii)$]. Thus $M$ is principally supplemented.
\end{proof}

\quad The following example shows that a principally $\mathcal{G}$*-lifting module need not be principally lifting in general:

\begin{example}
Consider the $\mathbb{Z}$-module $M=\mathbb{Z} / 2\mathbb{Z} \oplus \mathbb{Z}/ 8\mathbb{Z}$. By [\cite{Yongduo}, Example $3.7$], $M$ is $H$-supplemented module. Then $M$ is $\mathcal{G}$*-lifting by [\cite{Býrkenmeýer}, Theorem $3.6$]. Since every $\mathcal{G}$*-lifting module is principally $\mathcal{G}$*-lifting, $M$ is principally $\mathcal{G}$*-lifting. But from [\cite{Harmancý}, Examples $7.(3)$], $M$ is not principally lifting. \\
\end{example}

\begin{theorem}\label{ind}
Let $M$ be an indecomposable module. Consider the following conditions:
\begin{itemize}
  \item[(a)] $M$ is principally lifting,
  \item[(b)] $M$ is principally hollow,
  \item[(c)] $M$ is principally $\mathcal{G}$*-lifting.
\end{itemize}
Then $(a)\Leftrightarrow (b)\Leftrightarrow (c)$
\end{theorem}
\begin{proof}
$(a)\Leftrightarrow(b)$ It is easy to see from [\cite{Harmancý}, Lemma $14$].\\
$(b)\Rightarrow(c)$ Let $M$ be principally hollow and $m\in M$. Then $mR \ll M$ implies that $mR\beta$*$0$.\\
$(c) \Rightarrow (b)$ Let $mR$ be a proper cyclic submodule of $M$. By $(c)$, there exists a decomposition $M=D\oplus D'$ such that $mR \beta$*$D$. Since $M$ is indecomposable, $D=M$ or $D=0$. Let $D=M$. From [\cite{Býrkenmeýer}, Corollary $2.8.(iii)$], we obtain $mR=M$ but this is a contradiction. Thus  $mR\beta$*$0$ and so $mR\ll M$. That is, $M$ is principally hollow.
\end{proof}
We shall give the following example of modules which are principally supplemented but not principally $\mathcal{G}$*-lifting.
\begin{example}
Let $F$ be a field, $x$ and $y$ commuting indeterminates over $F$. Let $R= F[x,y]$ be a polynomial ring and its ideals $I_{1}=(x^{2})$ and $I_{2}=(y^{2})$ and the ring $S=R/(x^{2},y^{2})$. Consider the $S$-module $M=\overline{x}S+\overline{y}S$. By [\cite{Harmancý}, Example $15$], $M$ is an indecomposable $S$-module and it is not principally hollow. Then $M$ is not principally $\mathcal{G}$*-lifting from Theorem \ref{ind}. Again by [\cite{Harmancý}, Example $15$], $M$ is principally supplemented.
\end{example}

A module $M$ is said to be \emph{principally semisimple} if every cyclic submodule of $M$ is a direct summand of $M$.

\begin{lemma}
Every principally semisimple module is principally $\mathcal{G}$*-lifting.
\end{lemma}
\begin{proof}
It is clear from the definition of semisimple modules and the reflexive property of $\beta$*.
\end{proof}

Recall that a submodule $N$ is called \emph{fully invariant} if for each endomorphism $f$ of $M$, $f(N)\leq N$. A module $M$ is said to be \emph{duo module} if every submodule of $M$ is fully invariant. A module $M$ is called \emph{distributive} if for all submodules $A,B,C$ of $M$, $A+(B\cap C)=(A+B)\cap (A+C)$ or $A\cap (B+C)=(A\cap B)+(A\cap C)$.\\

\begin{proposition}
Let $M = M_{1}\oplus M_{2}$ be a duo module (or distributive module). Then $M$ is principally $\mathcal{G}$*-lifting if and only if $M_{1}$ and $M_{2}$ are principally $\mathcal{G}$*-lifting.
\end{proposition}
\begin{proof}
$(\Rightarrow)$ Let $m\in M_1$. Since $M$ is principally $\mathcal{G}$*-lifting, there is a decomposition $M=D\oplus D'$ such that $mR\beta$*$D$ in $M$. As $M$ is duo module $M_1=(M_1\cap D)\oplus (M_1\cap D')$, $mR=(mR\cap D)\oplus (mR\cap D')$ and \linebreak $D'=(M_1\cap D')\cap(M_2\cap D')$. We claim that $mR\beta$*$(M_1\cap D)$ in $M_1$. Since $\dfrac{mR+(M_1\cap D)}{mR}\leq\dfrac{mR+D}{mR}$ and $\dfrac{mR+D}{mR}\ll\dfrac{M}{mR}$, we have $\dfrac{mR+(M_1\cap D)}{mR}\ll\dfrac{M}{mR}$ by [\cite{Wisbauer}, $19.3(2)$]. From isomorphism theorem and the direct decomposition of $mR$
$$\dfrac{mR+(M_1\cap D)}{M_1\cap D}\cong\dfrac{mR}{mR\cap(M_1\cap D)}=\dfrac{mR}{mR\cap D}\cong mR\cap D'.$$ Since $D'$ is a supplement  of $mR$, $mR\cap D'\ll D'$. By [\cite{Wisbauer}, $19.3(5)$], \linebreak $mR\cap D'\ll M_1\cap D'$. Further $M_1\cap D'\cong\dfrac{M_1}{M_1\cap D}$. This shows that $\dfrac{mR+(M_1\cap D)}{M_1\cap D}$ is small in $\dfrac{M_1}{M_1\cap D}$ and also in $\dfrac{M}{M_1\cap D}$.  From Definition \ref{rel} we get $mR\beta$*$(M_1\cap D)$ in $M$. Then  $mR\beta$*$(M_1\cap D)$ in $M_1$ by Lemma \ref{lemma1}.

$(\Leftarrow)$ Let $m\in M$. If $M$ is a duo module, for cyclic submodule $mR$ of $M$, \linebreak $mR=(mR\cap M_{1})\oplus (mR\cap M_{2})$. So $mR\cap M_{1}=m_{1}R$ and $mR\cap M_{2}=m_{2}R$ for some $m_1\in M_1$, $m_2\in M_2$ . Since $M_{1}$ and $M_{2}$ are principally $\mathcal{G}$*-lifting, there are decompositions $M_1=D_1\oplus D_1'$ and $M_2=D_2\oplus D_2'$ such that $m_{1}R\beta$*$D_1$ in $M_1$ and $m_{2}R\beta$*$D_2$ in $M_2$. By Lemma \ref{lemma1}, $m_{1}R\beta$*$D_1$  and $m_{2}R\beta$*$D_2$ in $M$.  Since $mR=m_{1}R+m_{2}R$, by [\cite{Býrkenmeýer}, Proposition $2.11$], $mR\beta$*$(D_1\oplus D_2)$.\\
\end{proof}


\begin{proposition}
Let any cyclic submodule of $M$ have a supplement which is a relatively projective direct summand of $M$. Then $M$ is principally $\mathcal{G}$*-lifting.
\end{proposition}
\begin{proof} Let $m\in M$. By hypothesis, there exsists a decomposition \linebreak $M=D\oplus D'$ such that $M=mR+D'$ and $mR\cap D'\ll D'$. Because $D$ is $D'$-projective,  $M=A\oplus D'$ for some submodule $A$ of $mR$ by  [\cite{Muller}, Lemma $4.47$]. So $M$ is principally lifting. It follows from Theorem \ref{thm3} that M is principally $\mathcal{G}$*-lifting.
\end{proof}

\begin{proposition}\label{quo}
Let $M$ be principally  $\mathcal{G}$*-lifting and $N$ be a submodule of $M$. If $\dfrac{N+D}{N}$ is a direct summand in $\dfrac{M}{N}$ for any cyclic direct summand $D$ of $M$, then $\dfrac{M}{N}$ is principally $\mathcal{G}$*-lifting.
\end{proposition}
\begin{proof}
Let $\dfrac{mR+N}{N}$ be a cyclic submodule of $\dfrac{M}{N}$ for $m\in M$. Since $M$ is principally $\mathcal{G}$*-lifting there exists a decomposition $M=D\oplus D'$ such that $mR\beta$*$D$. Then $D$ is also cyclic Lemma \ref{lemma2}. By hypothesis, $\dfrac{D+N}{N}$ is a direct summand in $\dfrac{M}{N}$. We claim that $\dfrac{mR+N}{N}\beta$*$\dfrac{D+N}{N}$. Consider the canonical epimorphism $\theta : M\rightarrow M/N$. By [\cite{Býrkenmeýer}, Proposition $2.9(i)$], $\theta (mR)\beta$*$\theta (D)$, that is, $\dfrac{mR+N}{N}\beta$*$\dfrac{D+N}{N}$. Thus $\dfrac{M}{N}$ is principally $\mathcal{G}$*-lifting.
\end{proof}

\begin{corollary}
Let $M$ be principally  $\mathcal{G}$*-lifting.
\begin{itemize}
\item[(a)] If $M$ is distributive (or duo) module, then any quotient module of $M$ principally  $\mathcal{G}$*-lifting.
\item[(b)] Let $N$ be a projection invariant, that is, $eN\subseteq N$ for all $e^2=e\in End(M)$. Then $\dfrac{M}{N}$ is principally  $\mathcal{G}$*-lifting. In particular, $\dfrac{M}{A}$ is principally  $\mathcal{G}$*-lifting for every fully invariant submodule $A$ of $M$.
\end{itemize}
\end{corollary}
\begin{proof}
$(a)$ Let $N$ be any submodule of $M$ and $M=D\oplus D'$ for some submodules $D,D'$ of $M$. Then
$$\dfrac{M}{N}=\dfrac{D\oplus D'}{N}=\dfrac{D+N}{N}+\dfrac{D'+N}{N}.$$
Since $M$ is distributive, $N=(D+N)\cap(D'+N)$. We obtain $\dfrac{M}{N}=\dfrac{D+N}{N}\oplus \dfrac{D'+N}{N}$. By Theorem \ref{quo}, $\dfrac{M}{N}$ is principally  $\mathcal{G}$*-lifting.\\

$(b)$ Assume that $M=D\oplus D'$ for some $D,D'\leq M$. For the projection map $\pi_D:M\rightarrow D$, $\pi_D^2=\pi\in End(M)$ and $\pi_D(N)\subseteq N$. So $\pi_D(N)=N\cap D$. Similarly, $\pi_{D'}(N)=N\cap D'$ for the projection map $\pi_{D'}:M\rightarrow D'$. Hence we have $N=(N\cap D)+(N\cap D')$. By modularity,
$$M=[D+(N\cap D)+(N\cap D')]+(D'+N)=[D\oplus(N\cap D')]+(D'+N).$$
and
$$[D\oplus(N\cap D')]\cap(D'+N)=[D\cap(D'+N)]+(N\cap D')=(N\cap D)+(N\cap D')=N$$
Thus $\dfrac{M}{N}=\dfrac{D\oplus(N\cap D')}{N}\oplus\dfrac{D'+N}{N}$. By Theorem \ref{quo}, $\dfrac{M}{N}$ is principally $\mathcal{G}$*-lifting.
\end{proof}

Another consequence of Proposition $3.10$ is given in the next result.

A module $M$ is said to have the \emph{summand sum property} (SSP) if the sum of any two direct summands of $M$ is again a direct summand.

\begin{proposition}
Let $M$ be a principally $\mathcal{G}$*-lifting module. If $M$ has SSP, then any direct summand of $M$ is principally $\mathcal{G}$*-lifting.
\end{proposition}
\begin{proof}
Let $M=N\oplus N'$ for some submodule $N,N'$ of $M$. Take any cyclic direct summand $D$ of $M$. Since $M$ has SSP, $M=(D+N')\oplus T$ for some submodule $T$ of $M$. Then $$\dfrac{M}{N'}=\dfrac{D+N'}{N'}+\dfrac{T+N'}{N'}$$  By modularity, $$(D+N')\cap(T+N') = N'+[(D+N')\cap T]= N'.$$ So we obtain $$\dfrac{M}{N'}=\dfrac{D+N'}{N'}\oplus\dfrac{T+N'}{N'}.$$ Thus $N$ is principally $\mathcal{G}$*-lifting from Proposition \ref{quo}.
\end{proof}


Next we show when $M/ Rad(M)$ is principally semisimple in case $M$ is principally $\mathcal{G}$*-lifting module.
\begin{proposition}
Let $M$ be principally $\mathcal{G}$*-lifting and distributive module. Then $\dfrac{M}{Rad(M)}$ is a principally semisimple.
\end{proposition}
\begin{proof}
Let $m\in M$. There exists a decomposition $M=D\oplus D'$ such that $mR\beta$*$D$. By [\cite{Býrkenmeýer}, Theorem $2.6 (ii)$], $D'$ is a supplement of $mR$, that is, $M=mR+D'$ and $mR\cap D'\ll D'$. Then
$$\dfrac{M}{Rad(M)}=\dfrac{mR+Rad(M)}{Rad(M)} + \dfrac{D'+Rad(M)}{Rad(M)}.$$
Because $M$ is distributive and $mR\cap D'\ll D'$, $$(mR+Rad(M)) \cap (D'+ Rad(M))= (mR \cap D') + Rad(M)= Rad(M).$$  Hence $\dfrac{mR+Rad(M)}{Rad(M)}$ is a direct summand in $\dfrac{M}{Rad(M)}$, this means that $\dfrac{M}{Rad(M)}$ is a principally semisimple module.
\end{proof}

\begin{proposition}\label{pro1}
Let $M$ be a principally $\mathcal{G}$*-lifting module and \linebreak $Rad(M)\ll M$. Then $\dfrac{M}{Rad(M)}$ is principally semisimple.
\end{proposition}
\begin{proof}
Let $\dfrac{X}{Rad(M)}$ be a cyclic submodule of $\dfrac{M}{Rad(M)}$. Then $X=mR+Rad(M)$ for some $m\in M$. There exists a decomposition $M=D\oplus D'$ such that $mR\beta$*$D$. By [\cite{Býrkenmeýer}, Theorem $2.6 (ii)$], $D'$ is a supplement of $mR$ in $M$. It follows from [\cite{Býrkenmeýer}, Corollary $2.12$] that $(mR+Rad(M))\beta$*$D$. Therefore $$\dfrac{M}{Rad(M)}=\dfrac{X}{Rad(M)}+\dfrac{D'+Rad(M)}{Rad(M)}.$$ By modularity and $X\cap D'\subseteq Rad(M)$, $$\dfrac{X}{Rad(M)} \cap \dfrac{D'+Rad(M)}{Rad(M)}=\dfrac{(X\cap D')+Rad(M)}{Rad(M)}.$$ Then we obtain that  $\dfrac{M}{Rad(M)}=\dfrac{X}{Rad(M)}\oplus\dfrac{D'+Rad(M)}{Rad(M)}$.\\
\end{proof}

\begin{proposition}
Let $M$ be principally $\mathcal{G}$*-lifting. If $Rad(M)\ll M$, then $M=A\oplus B$ where $A$ is principally semisimple and $Rad(B) \ll B$.
\end{proposition}%
\begin{proof}
Let $A$ be a submodule of $M$ such that $Rad(M)\oplus A$ is small in $M$ and $m\in A$. By hypothesis, there exists a decomposition $M=D\oplus D'$ such that $mR\beta^*D$. Then $D$ is cyclic from Lemma \ref{lemma2}. By [\cite{Býrkenmeýer}, Theorem $2.6(ii)$], $M=mR+D'$ and $mR\cap D'\ll D'$. So $mR\cap D'=0$ because $mR\cap D'$ is a submodule of $Rad(M)$. That is, $M=mR\oplus D'$. Hence $mR=D$. Since $D\cap Rad(M)=0$, $D$ is isomorphic to a submodule of $\dfrac{M}{Rad(M)}$. By Proposition \ref{pro1}, $\dfrac{M}{Rad(M)}$ is principally semisimple. Thus $D$ is principally semisimple.
\end{proof}
In general, it is not true that principally lifting and principally $\mathcal{G}$*-lifting modules coincide. As we will see the following theorem, we need projectivity condition.
\begin{proposition}\label{pro}
Let $M$ be a $\pi$-projective module. The following are equivalent:
\begin{itemize}
  \item[(a)] $M$ is principally lifting,
  \item[(b)] $M$ is principally $\mathcal{G}$*-lifting,
  \item[(c)] $M$ is $\oplus$-principally supplemented.
\end{itemize}
\end{proposition}
\begin{proof}
$(a)\Rightarrow (b)$ follows from Theorem \ref{thm3}.\\
$(b)\Rightarrow(c)$  follows from [\cite{Býrkenmeýer}, Theorem $2.6(ii)$].\\
$(c)\Rightarrow(a)$ Consider any $m\in M$. From $(c)$, $mR$ has a supplement $D$ which is a direct summand in $M$, that is, $M=mR+D$ and $mR\cap D\ll D$. Since $M$ is $\pi$-projective there exists a complement $D'$ of $D$ such that $D'\subseteq mR$ by [\cite{Clark}, $4.14(1)$]. Thus $M$ is principally lifting.
\end{proof}

\begin{proposition}
Let $M$ be a $\pi$-projective module. Then $M$ is principally $\mathcal{G}$*-lifting if and only if every cyclic submodule $X$ of $M$ can be written as $X=D\oplus A$ such that $D$ is a direct summand in $M$ and $A\ll M$.
\end{proposition}
\begin{proof}
$(\Rightarrow)$ Let $M$ be principally $\mathcal{G}$*-lifting and $\pi$-projective module. By Proposition \ref{pro}, $M$ is principally lifting. Then for any cyclic submodule $X$ of  $M$ there exists a direct decomposition $M=D\oplus D'$ such that $D\leq X$ and $X \cap D'\ll M$. By modularity, we conclude that $X=D\oplus (X\cap D')$.
\newline $(\Leftarrow)$ By assumption and [\cite{Kamal}, Lemma $2.10$], $M$ is principally lifting. Hence from Proposition \ref{pro} $M$ is principally $\mathcal{G}$*-lifting.
\end{proof}

Now we mention that principally $\mathcal{G}$*-lifting and $\mathcal{G}$*-lifting modules are coincide under some condition.

\begin{proposition}
Let $M$ be Noetherian and have SSP. Then $M$ is principally $\mathcal{G}$*-lifting if and only if $M$ is $\mathcal{G}$*-lifting.
\end{proposition}
\begin{proof}
$(\Leftarrow)$ Clear.\\
$(\Rightarrow)$ If $M$ is Noetherian, for any submodule $X$ of $M$ there exist some $m_{1},m_{2},...,m_{n}\in M$ such that $X=m_{1}R+m_{2}R+...+m_{n}R$ by [\cite{Wisbauer}, $27.1$].Since $M$ is principally $\mathcal{G}$*-lifting, there exist some direct summands $D_{1}, D_{2},...,D_{n}$ of $M$ such that $m_{1}R\beta$*$D_{1}$, $m_{2}R\beta$*$D_{2}$,...,$m_{n}R\beta$*$D_{n}$. $D=D_{1}+D_{2}+...+D_{n}$ is also a direct summand in $M$ because of SSP. By [\cite{Býrkenmeýer}, Proposition $2.11$], $X\beta$*$D$. Hence $M$ is $\mathcal{G}$*-lifting.
\end{proof}

\begin{proposition}
Let any submodule $N$ of $M$ be  the sum of a cyclic submodule $X$ and a small submodule $A$ in $M$. Then $M$ is principally $\mathcal{G}$*-lifting if and only if $M$ is $\mathcal{G}$*-lifting.
\end{proposition}
\begin{proof}
$(\Leftarrow)$ Clear.\\
$(\Rightarrow)$ Let $N=X+A$ for a cyclic submodule $X$ and a small submodule $A$ of $M$. Since $M$ is principally $\mathcal{G}$*-lifting, there exists a direct summand $D$ of $M$ such that $X\beta^*D$. From [\cite{Býrkenmeýer}, Corollary $2.12$], $(X+A)\beta^*D$, that is, $N\beta^*D$. Hence $M$ is $\mathcal{G}$*-lifting.
\end{proof}

\end{document}